\documentclass[11pt,reqno]{amsart}
\usepackage{graphicx}

\usepackage{amssymb,amsmath,amsthm,amscd,latexsym,amsfonts}
\usepackage{mathtools}
\usepackage[T1]{fontenc}
\usepackage[arrow,matrix]{xy}
\usepackage{graphicx}
\usepackage{color}
\usepackage{cite}
\usepackage{tikz}
\newtheorem{thm}{Theorem}
\newtheorem{defn}{Definition}
\newtheorem{lem}{Lemma}
\newtheorem{prop}{Proposition}
\newtheorem{rem}{Remark}

\numberwithin{equation}{section} 

\begin{document}

\title{Dynamics of a population with two equal dominated species}

\author[U.\,A.\,Rozikov, J.\,B. \,Usmonov]{U.~A.~Rozikov, J.~B.~Usmonov}

 \address{U.\,A.\,Rozikov\\ Institute of Mathematics,
81, Mirzo Ulug'bek str., 100170, Tashkent, Uzbekistan.}
\email{rozikovu@yandex.ru}

 \address{J.\,B.\,Usmonov\\ Institute of Mathematics,
81, Mirzo Ulug'bek str., 100170, Tashkent, Uzbekistan.}
\email{javohir0107@gmail.com}

\begin{abstract} We consider a population with two equal dominated species, dynamics of which
is defined by an one-dimensional piecewise-continuous, two parametric function.
It is shown that for any non-zero parameters this function has two fixed points and several periodic points.
We prove that all periodic (in particular fixed) points are repelling, and find an invariant set
which asymptotically involves the trajectories of any initial point except fixed and periodic ones.
We showed that the orbits are unstable and chaotic because Lyapunov exponent is non-negative.
The limit sets analyzed by bifurcation diagrams. We give biological interpretations of our results.

\end{abstract}



\maketitle

\section{Introduction}

Unlike Markov processes, where the probability distribution of the system evolves in a
linear fashion under the action of a stochastic operator, population dynamics
are nonlinear because the recombination of genes occurs due to their pairing. The mathematical
model of population dynamics was given in seminal works of Kesten \cite{Ke}, where he
studied dynamical systems generated by
  quadratic stochastic operator (QSO) of asexual (multi-type) populations and sex-linked systems.
He gave conditions on coefficients of QSO under which it has a unique fixed point. Moreover, several results
are given for the different mating rules and a stochastic theory for the Mendelian genetics model is given.
Kesten's papers form a valuable contribution to the mathematics of population processes. Future development
of the theory of QSO made by many authors (see for example \cite{G2}, \cite{G4}, \cite{GMR}, \cite{GGJ},
\cite{L}, \cite{M}, \cite{ME}, \cite{RS1}, \cite{RZ}, \cite{RSh}, \cite{RJ} and the references therein).

 In \cite{Ke} the random process behind the dynamical system of the type distribution
 (i.e. the random sizes of each type) is investigated;
 but, as he admitted himself, the results are somewhat disappointing as the model only
 yields a dichotomy between extinction and exponential growth, thus not demonstrating
 stability or adaptation (selection). In this paper we give a model of population
 where the dichotomy between extinction and exponential growth is \emph{not} the case.
 Our model is given by a function with one discontinuity point. We show that
 the dynamical system has a complexity despite the discontinuity point is unique.

 Note that many dynamical systems that happen naturally in the description of physical and biological
 processes are piecewise-smooth.  That's why the dynamical system containing terms that are non-smooth
 functions (as in our case) is studied as an important problem. Problems of like this appear, for example,
  in electrical circuits that have switches, mechanical devices in which components impact with each other or
  have free play, problems with friction, sliding or squealing, many control systems
   and models in the social and financial sciences where continuous change can trigger discrete actions \cite{ref2}.
See also motivating examples of a piecewise-smooth systems: generated by the floor function (\cite{ref1}, \cite{ref4}) and
 $p$-adic dynamical systems (see \cite{ARS}, \cite{LRS}, \cite{RS}, \cite{RSY}).

 Thus in this paper we consider one of such example,
 which arise in population biology as an evolution operator (piecewise-continuous Volterra QSO).
 
 The paper is organized as follows. In Section 2 we give necessary definitions and an evolution operator
 of the population with two dominating species. In Section 3 reducing the evolution operator to a
 function (with unique discontinuity point $1/2$) defined on $[0,1]$ we study its fixed and periodic points.
 Section 4 is devoted to Lyapunov exponents and bifurcations of the dynamical systems. The last section contains some
 biological interpretations of our results.

 \section{Definitions}

{\it Basic definitions}. Let us give some necessary definitions (see chapter 1 of \cite{Rb}).
 In order to define a discrete-time dynamical system consider a function $f:X\to X$.

For $x\in X$ denote
by $f^n(x)$ the $n$-fold composition of $f$ with itself (i.e. $n$ time iteration of $f$ to $x$):
$$f^n(x)=\underbrace{f(f(f\dots (f}_{n \,{\rm times}}(x)))\dots).$$

\begin{defn}\label{DS} For arbitrary given $x_0\in X$ and $f:X\to X$ the discrete-time
dynamical system (also called forward orbit or trajectory of $x_0$) is the sequence of points
\begin{equation}\label{eds}
x_0, x_1=f(x_0), x_2=f^2(x_0), x_3=f^3(x_0), \dots
\end{equation}
\end{defn}

\begin{defn}\label{fp}
A point $x\in X$ is called a \index{fixed point} fixed point for $f:X\to X$ if $f(x)=x$. The point $x$
is a periodic point of period $p$ if $f^p(x) = x$. The least positive $p$ for which
$f^p(x) = x$ is called the \index{prime period} prime period of $x$.
\end{defn}

Denote the set of all fixed
points by Fix$(f)$ and the set of all periodic
points of (not necessarily prime) period $p$ by ${\rm Per}_p(f)$.

It is clear that the set of all iterates of a periodic point form a periodic
sequence (orbit).

  There are three kinds of periodic points: attracting, repelling and indifferent.
  Let $x^\ast$ be a $p$-periodic point. If $|(f^p(x^\ast))'|<1$, $x^\ast$ - attracting; $|(f^p(x^\ast))'|>1$, $x^\ast$ - repelling; $|(f^p(x^\ast))'|=1$, $x^\ast$ - indifferent.

{\it Evolution operator.}
Let $S^{m-1}$ be the simplex:
$$S^{m-1}=\{x=(x_1, \dots, x_m)\in \mathbb R^m: x_i\geq 0, \ \sum_{i=1}^mx_i=1\}.$$

Consider a population consisting of $m$ species. Let $x^{0} =
(x_{1}^{0},\dots,x_{m}^{0})\in S^{m-1}$ be the probability distribution (where
$x^{0}_i=P(i)$ is the probability of $i,\, i=1,2,\dots,m$) of
species in the initial generation, and $P_{ij,k}$ the probability
that individuals in the $i$th and $j$th species interbreed to
produce an individual $k$, more precisely $P_{ij,k}$ is the
conditional probability $P(k|i,j)$ that $i$th and $j$th species
interbred successfully, then they produce an individual $k$.

Assume the ``parents" $ij$ are
independent i.e., $P(i,j)=P(i)P(j)=x^0_ix^0_j$. Then the probability
distribution $x'= (x_{1}',\dots,x_{m}')$ (the state) of the species
in the first generation can be found by the total probability
\begin{equation}\label{j1}
 x'_k= \mathop {\sum}
\limits^{m}_{i,j=1}P(k|i,j)P(i,j)=  \mathop {\sum}
\limits^{m}_{i,j=1}P_{ij,k}x^{0}_{i}x^{0}_{j} ,\,\,\,\,k= 1,\dots,m.
\end{equation}
 This means that the association $x^{0}\in S^{m-1}\rightarrow x'\in S^{m-1}$ (i.e. (\ref{j1}))
defines a map $V$ called the \textit{evolution operator}.

The states of the population described by the following discrete-time dynamical
 system
  \begin{equation}\label{2}
 x^{0},\ \ x'= V(x^0), \ \ x''=V^{2}(x^0),\ \  x'''= V^{3}(x^0),\dots
\end{equation}
where $V^n(x)=\underbrace{V(V(...V}_n(x))...)$
denotes the $n$ times iteration of $V$ to $x$.

{\bf The main problem} for a given dynamical system is to
describe the limit points of $\{x^{(n)}\}_{n=0}^\infty$ for
arbitrary given $x^{(0)}$.

The difficulty of the problem depends on the given QSO $V$,
see \cite{GMR}, \cite{L} and \cite{MG} for the results on this main problem.
Note that the known results are mainly shows that the corresponding to a QSO
population has behavior to make a dichotomy between extinction and exponential growth.
In this paper our aim is to consider a model of population which does not have such behavior.

{\it The model}.  Consider a population consisting of two species, i.e. $m=2$.
Denote the set of species by $E=\{1,2\}.$

For a parameter
$a\in [-1,1]$ define the operator $V_a: S^1\to S^1$ as
$$V_a : \left\{\begin{array}{ll}
x'=x(1+ay)\\[2mm]
y'=y(1-ax)
\end{array}\right.$$
For an initial point $z=(x,y)\in S^1$ the trajectory $z^{(n)}=(x^{(n)}, y^{(n)})=V^n(z)$ is given by
$$\left\{\begin{array}{ll}
x^{(n+1)}=x^{(n)}(1+ay^{(n)})\\[2mm]
y^{(n+1)}=y^{(n)}(1-ax^{(n)})
\end{array}\right.$$
From this system it is clear that the sequences $x^{(n)}$ and $y^{(n)}$
are monotone for any $a\ne 0$ (the case
$a=0$ gives a trivial dynamical system, therefore we will not consider it).
Since both sequences are bounded they have a limit, the limit points are fixed points of $V_a$. Therefore we get the following
\begin{equation}\label{j3}
\lim_{n\to \infty} z^{(n)}=\left\{\begin{array}{ll}
(0,1), \ \ \mbox{if} \ \ a<0\\[2mm]
(1,0), \ \ \mbox{if} \ \ a>0
\end{array}\right.
\end{equation}
Thus (\ref{j3}) means that if $a>0$ (resp. $a<0$) then the specie 2 (resp. 1) will extinct and
the specie 1 (resp. 2) will dominate (grow).

To ensure that both species will have equal domination we define an evolution operator,
$V_{a,b}:z=(x,y)\in S^1\to z'=(x',y')\in S^1$ by

\begin{equation}\label{vab}V_{a,b} = \left\{\begin{array}{ll}
V_a(z), \ \ \mbox{if} \ \ x\leq {1\over 2}\\[2mm]
V_b(z), \ \ \mbox{if} \ \ x> {1\over 2}
\end{array}\right.
\end{equation}

This is a piecewise-continuous operator, in the next sections we show that
the dynamical system generated by $V_{a,b}$ for some conditions on $a,b\in [-1,1]$
has property that both species will always survive.

We note that the probabilities $P_{ij,k}$ mentioned in (\ref{j1}) are
independent on $x\in S^{m-1}$, but for operator (\ref{vab}) we have
$$P_{11,1}=1-P_{11,2}=P_{22,2}=1-P_{22,1}=1$$
and the remaining coefficients
depend on the points $z=(x,y)$ of the simplex $S^1$:
\begin{equation}\label{vabp}P_{12,1}(z)=1-P_{12,2}(z)= \left\{\begin{array}{ll}
{1+a\over 2}, \ \ \mbox{if} \ \ x\leq {1\over 2}\\[2mm]
{1-b\over 2}, \ \ \mbox{if} \ \ x> {1\over 2}.
\end{array}\right.
\end{equation}

\section{A function with unique discontinuity point}

Consider the dynamical system generated by the evolution operator $V_{a,b}$. Using the equality $x+y=1$, this operator
can be reduced to the function $f_{a,b}:[0,1]\rightarrow[0,1]$ defined by
\begin{equation}\label{mfunc}
f_{a,b}(x)=\left\{
        \begin{array}{ll}
          x(1+a-ax), & \hbox{$0\leq x\leq \frac{1}{2}$;} \\
          x(1-b+bx), & \hbox{$\frac{1}{2}<x\leq 1$}
        \end{array}
      \right.
\end{equation}
where by the symmetry of parameters we can assume that $a,b\in[0,1]$.

It is clear that, the function is a piecewise-continuous, that's, it is discontinuous at the point $x=\frac{1}{2}$ when $(a,b)\ne(0,0)$ and, is smooth at each semi interval.

The notion of topological conjugacy is important in the study of iterated functions and more generally dynamical systems, since, if the dynamics of one iterated function can be solved, then those for any topologically conjugate function follow trivially. We give its definition.

\begin{defn}\label{def1}\cite{ref3} Let $f:A\rightarrow A$ and $g:B\rightarrow B$ be two maps. $f$ and $g$ are said to be topologically conjugate if there exists a homeomorphism $h:A\rightarrow B$ such that, $h\circ f=g\circ h$. The homeomorphism $h$ is called a topologically conjugacy.
\end{defn}

\begin{prop}\label{lem1}
$f_{a,b}$ and $f_{\tilde{a},\tilde{b}}$ are topologically conjugate if and only if $\tilde{a}=b$ and $\tilde{b}=a$.
\end{prop}
\begin{proof} Let $\tilde{a}=b$ and $\tilde{b}=a$. According to the Definition \ref{def1}, it is necessary to find $h$ homeomorphism satisfying the condition $h\circ f_{a,b}=f_{b,a}\circ h$. We take the function $h(x)=1-x$ as the homeomorphism $h$. Let's check the above mentioned equality:
$$h(f_{a,b}(x))=\left\{
     \begin{array}{ll}
       1-x(1+a-ax), & \hbox{$0\leq x\leq \frac{1}{2}$;} \\
       1-x(1-b+bx), & \hbox{$\frac{1}{2}<x\leq1$.}
     \end{array}
   \right.$$

$$f_{b,a}(h(x))=\left\{
     \begin{array}{ll}
       (1-x)(1+b-b(1-x)), & \hbox{$0\leq 1-x\leq \frac{1}{2}$;} \\
       (1-x)(1-a+a(1-x)), & \hbox{$\frac{1}{2}<1-x\leq1$.}
     \end{array}
   \right.=$$
   $$=\left\{
     \begin{array}{ll}
       1-x(1-b+bx), & \hbox{$\frac{1}{2}\leq x\leq 1$;} \\
       1-x(1+a-ax), & \hbox{$0\leq x <\frac{1}{2}$.}
     \end{array}
   \right.= $$
   $$=\left\{
     \begin{array}{ll}
       1-x(1+a-ax), & \hbox{$0\leq x <\frac{1}{2}$;} \\
       1-x(1-b+bx), & \hbox{$\frac{1}{2}\leq x\leq 1$.}
     \end{array}\right.=1-f_{a,b}(x)=h(f_{a,b}(x))$$
\end{proof}
\begin{rem}
According to Proposition \ref{lem1} it is sufficient
to study the dynamical system generated by (\ref{mfunc}) in the domain $a\in [0,1]$, $a\leq b$.
\end{rem}
For $a\leq b$ we consider the following possible cases:
\begin{enumerate}

\item $a=b=0$;
\item  $a=0,\, b\ne0$;
\item $a\ne0,\, b\ne0$.

\end{enumerate}

\begin{rem}
 If $a=b=0$ then the function has the form $f_{0,0}(x)=x$,\, $0\leq x \leq1$. This case is not interesting.
\end{rem}

\subsection{The case $a=0$, $b\ne0$}
Then the function is
\begin{equation}\label{pfunc1}
f_{0,b}(x)\equiv f(x)=\left\{
     \begin{array}{ll}
       x, & \hbox{$0\leq x\leq \frac{1}{2}$;} \\
       x(1-b+bx), & \hbox{$\frac{1}{2}<x\leq1$.}
     \end{array}
   \right.
\end{equation}

\begin{prop}\label{pj1}
For the dynamical system generated by function (\ref{pfunc1}) the following hold:
\begin{itemize}
\item[1)] {\rm Fix}$(f)=[0,\frac{1}{2}]\bigcup\{1\}$;
\item[2)] if $x\in (\frac{1}{2},1)$, then there exists $n\in N$ such that
$$f^{n}(x)=p,\ \ f^{n+1}(x)=f(p)=p,$$
where $p\in(\frac{1}{2}-\frac{b}{4},\frac{1}{2}]$;
\item[3)] for any initial point $x_0\in (\frac{1}{2}-\frac{b}{4}, \frac{1}{2}]$ the following recurrence formula expresses the orbit of the points which tends to $x_0$:
$$x_{n+1}=\frac{1}{b}\left(\sqrt{bx_n+\left(\frac{1-b}{2}\right)^2}+\frac{b-1}{2}\right),\, n\geq0.$$
\end{itemize}
\end{prop}

\begin{proof}  1) Follows from a simple analysis of the equation $f(x)=x$.

 2) The function $f(x)$  has properties: $f(x)<x$ and $f(x)>\frac{1}{2}-\frac{b}{4}$ at any point $x\in (\frac{1}{2},1)$.  If $f(x)>\frac{1}{2}$, then $f^2(x)<x$ and so on. So, for any $x\in (\frac{1}{2},1)$ there exists $n \in N$, such that $f^n(x)<\frac{1}{2}$. That's, $f^{n+1}(x)=f(f^n(x))=f(p)=p$, for some  $p\in(\frac{1}{2}-\frac{b}{4},\frac{1}{2}]$ (see Figure \ref{fig:pfunc}).

 3) For proving we find the set of a preimage of $x_0\in (\frac{1}{2}-\frac{b}{4}, \frac{1}{2}]$ with respect to $f(x)$. To do this, we need to solve $f(x_{n+1})=x_n$ for $x_{n+1}\in [0,1]$. As a result, we get $x_{n+1}=\frac{1}{b}\left(\sqrt{bx_n+\left(\frac{1-b}{2}\right)^2}+\frac{b-1}{2}\right)$, $n\geq0$.
\end{proof}

\begin{figure}
\begin{center}
 \includegraphics[scale=0.75]{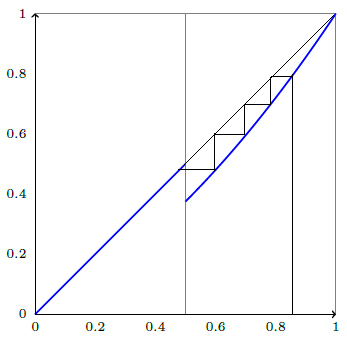}
\caption{The graph and phase portraits of (\ref{pfunc1}) with $b=\frac{1}{2}$.}
\label{fig:pfunc}
\end{center}
\end{figure}

\subsection{The case $a\ne0$, $b\ne0$}

In this case, for simplicity of formulas,  we omit indexes, i.e.
  \begin{center}
$f_{a,b}(x)\equiv f(x)=\left\{
        \begin{array}{ll}
          x(1+a-ax), & \hbox{$0\leq x\leq \frac{1}{2}$;} \\
          x(1-b+bx), & \hbox{$\frac{1}{2}<x\leq 1$.}
        \end{array}
      \right.$
\end{center}
\begin{figure}
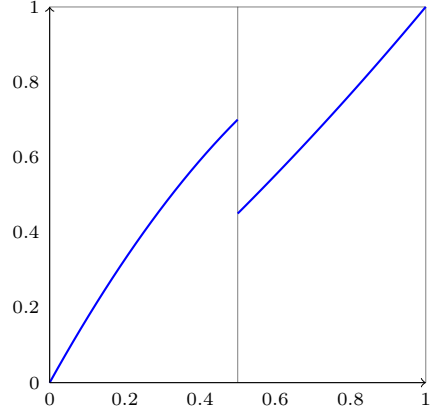

\begin{center}
\tikz[yscale=5,xscale=5]{
\draw [help lines,xstep=0.5cm] (0,0) grid (1,1);
\draw [blue,thick] (0,0) plot [domain=0.5:1] (\x,\x-0.2*\x+0.2*\x*\x);
\draw [blue,thick] (0,0) plot [domain=0:0.5] (\x,\x+0.8*\x-0.8*\x*\x);
\draw [->] (0,0) -- (1,0) ;
\draw [->] (0,0) -- (0,1) ;
\foreach \x in {0,0.2,0.4,0.6,0.8,1}{\node [left] at (0,\x) {\tiny \x};}
\foreach \y in {0,0.2,0.4,0.6,0.8,1}{\node [below] at (\y,0) {\tiny \y};}
}
\end{center}
\caption{The graphics of (\ref{mfunc}) with $a=\frac{1}{5}$, $b=\frac{4}{5}$.}
\end{figure}

\begin{prop}\label{prop2} Let $A=(\frac{1}{2}-\frac{b}{4},\frac{1}{2}+\frac{a}{4}]$ be subset of $[0,1]$. Then $f(A)=A$.
\end{prop}

\begin{proof} First of all we prove $f(x)\in A$ for all $x\in A$. For $x\in A$ we have $x\in (\frac{1}{2}-\frac{b}{4},\frac{1}{2}]$ or $x\in (\frac{1}{2},\frac{1}{2}+\frac{a}{4}]$. In both cases we prove that $f(x)\in A$. Let's suppose, $x\in (\frac{1}{2}-\frac{b}{4},\frac{1}{2}]$. It is easy to check that, $f(\frac{1}{2}-\frac{b}{4}),f(\frac{1}{2})\in A$. Then $f(x)\in A$ because $f(x)$ is monotonically increasing. For the second case, i.e., $x\in (\frac{1}{2},\frac{1}{2}+\frac{a}{4}]$ we may check similarly. Also $f^{-1}((\frac{1}{2}-\frac{b}{4},\frac{1}{2}])\cup f^{-1}((\frac{1}{2},\frac{1}{2}+\frac{a}{4}])=A$. Therefore, $f(A)=A$.
\end{proof}

\begin{lem}\label{lj1}
For any $x\in [0,1]\setminus A$ there exists $n_0(x)\in N$, such that $f^{n_{0}(x)}(x)\in A$.
\end{lem}

The proof of Lemma follows from the properties that the function is convex in $\left[0,\frac{1}{2}\right]$, is concave in $\left(\frac{1}{2},1\right]$ and monotone in each interval.

\begin{rem}
    Obviously, the set of fixed points is $\rm{Fix}(f)=\{0,1\}$ for $ab\ne 0$. Besides we have $|f'(0)|=1+a>1, |f'(1)|=1+b>1$.
    Thus {\it both} fixed points are repeller. Therefore both species will always survive.
\end{rem}
\begin{thm}\label{tj1}
The dynamical system generated by the function (\ref{mfunc}) has 2-periodic points if and only if the parameters $a$, $b$ satisfy the following conditions
\begin{equation}\label{con1}
  a\in (0,1), \, b< \frac{a}{1-a}
\end{equation}.
\end{thm}

\begin{proof} \textit{Necessity.} Let's suppose that there are 2-periodic points, i.e.,  the following equation has roots:

\begin{equation}\label{f1}
\frac{f(f(x))-x}{f(x)-x}=0
\end{equation}

The 2-fold composition of $f(x)$ with itself is
$$f(f(x))=\left\{
           \begin{array}{ll}
             x(1+a-ax)(1+a-ax(1+a-ax)), & \hbox{$x, f(x)\in[0,\frac{1}{2}]$;} \\[2mm]
             x(1+a-ax)(1-b+bx(1+a-ax)), & \hbox{$x\in[0,\frac{1}{2}], f(x)\in(\frac{1}{2},1]$;} \\[2mm]
             x(1-b+bx)(1+a-ax(1-b+bx)), & \hbox{$x\in(\frac{1}{2},1], f(x)\in[0,\frac{1}{2}]$;} \\[2mm]
             x(1-b+bx)(1-b+bx(1-b+bx)), & \hbox{$x, f(x)\in(\frac{1}{2},1]$.}
           \end{array}
         \right.
$$
By solving equation (\ref{f1}) it's easy to check that there is no periodic points in cases 1 and 4 of this equality.

Now we need solve the equation (\ref{f1}) for cases 2 and 3.\\
 For the 2nd case ($x\in(0,\frac{1}{2}], f(x)\in(\frac{1}{2},1)$);\\

 $$\frac{x(1+a-ax)(1-b+bx(1+a-ax))}{x(1+a-ax)-x}=0 \Rightarrow $$
 $$abx^2-(2b+ab)x+b+\frac{b}{a}-1=0 \Rightarrow$$
$$x_{1}=\frac{ab+2b+\sqrt{ab(ab+4)}}{2ab}, \  \ x_{2}=\frac{ab+2b-\sqrt{ab(ab+4)}}{2ab}$$

 For the root $x_{1}$ we have $x_{1}=\frac{1}{2}+\frac{2b+\sqrt{ab(ab+4)}}{2ab}\geq \frac{1}{2}$, that's why it's not a 2-periodic point.
 If
\begin{equation}\label{f2}
  \frac{a}{a+1}< b\leq \frac{4a}{4-a^{2}}
\end{equation}
then $0< x_{2} \leq \frac{1}{2}$. Since $f(x)\in(\frac{1}{2},1)$ for $f(x_{2})=\frac{1}{2}+\frac{-2a+\sqrt{ab(ab+4)}}{2ab}$ we get
\begin{equation}\label{f3}
  \frac{b}{b+1}< a< \frac{4b}{4-b^{2}}.
\end{equation}

So if the conditions (\ref{f2}) and (\ref{f3}) hold then $\text{Per}_{2}(f)=\{x_{2},f(x_{2})\}$. It is easy to see that conditions (\ref{f2}) and (\ref{f3}) are equivalent to (\ref{con1}).

\textit{Sufficiency.} If the parameters $a$ and $b$ satisfy (\ref{con1}) then the equation (\ref{f1}) has roots. And these roots are 2-periodic points.
\end{proof}

\begin{thm}\label{thm1}If $f$ has a periodic point then the point is a repelling.
\end{thm}

\begin{proof} Let's suppose $x_{0},x_{1},...,x_{n-1}$ with $x_{i}=f^{i}(x_{0})$, that's these points lie on a cycle of period $n$ for $f$. Then according to the Chain Rule we have
$$(f^{n})'(x_{0})=(f\circ f^{n-1})'(x_{0})$$
$$=f'(f^{n-1}(x_{0}))(f^{n-1})'(x_{0})$$
$$=f'(x_{n-1})(f\circ f^{n-2})'(x_{0})$$
$$=f'(x_{n-1})f'(f^{n-2}x_{0})(f^{n-2})'(x_{0})$$
$$=f'(x_{n-1})f'(x_{n-2})(f\circ f^{n-3})'(x_{0})$$
$$\ldots, \ldots, \ldots$$
$$=f'(x_{n-1})f'(x_{n-2})\ldots f'(x_{2})f'(x_{1})f'(x_{0}).$$
We have
$$f'(x)=\left\{
          \begin{array}{ll}
            1+a-2ax, & 0\leq x< \frac{1}{2}; \\[3mm]
            does\, not \, exist, & x=\frac{1}{2}\ \ and\ \ f'(\frac{1}{2}-0)=1,\ \ f'(\frac{1}{2}+0)=-\infty;\\[3mm]
            1-b+2bx, & \frac{1}{2}<x\leq 1.
          \end{array}
        \right.$$
So, we have $0\leq x< \frac{1}{2}\Rightarrow 0\geq -2ax>-a\Rightarrow1+a\geq1+a-2ax>1$,
i.e., $1+a\geq f'(x)>1$, when $0\leq x<\frac{1}{2}$. Similarly $1+b\geq f'(x)>1$, when $\frac{1}{2}<x\leq 1$. Therefore, $|f'(x)|>1$ for all $x\in[0,1]\setminus\frac{1}{2}$. Since $f'(x)>1$ for any $x\in[0,1]\setminus \frac{1}{2}$ we have  $|(f^{n})'(x_{0})|>1$.

\end{proof}

Let's consider the following cases:

\begin{equation}\label{2-2}
a\in(0,1], \ \ \ a\leq b\leq \frac{4a}{4-a^2}
\end{equation}


\begin{equation}\label{2-1}
a\in(0,1], \ \ \ b> \frac{4a}{4-a^2}
\end{equation}


 In these cases the dynamical system has several properties which showed on graphics(only in the invariant set $A$):
 \begin{figure}[h]
\begin{center}
\includegraphics[width=6cm]{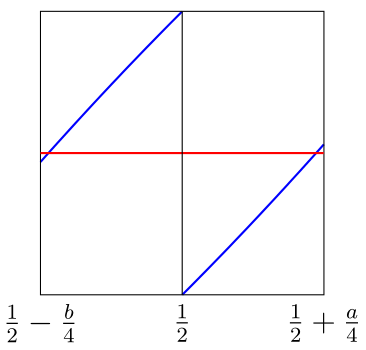}
\end{center}
\caption{The graphics of $y=1/2$ (red) and the function (\ref{mfunc}) on the invariant set $A$ when the condition (\ref{2-2}) holds.}
\label{fig:cross}
\end{figure}
\begin{figure}[h]
\begin{center}
\includegraphics[width=8cm]{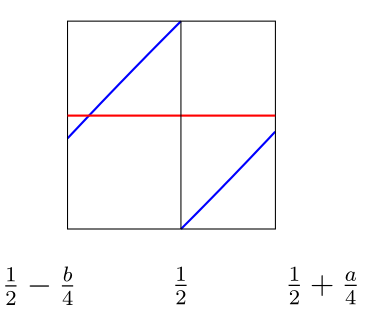}
\end{center}
\caption{The graphics of the function (\ref{mfunc}) in the invariant set $A$ when the condition (\ref{2-1}) holds.}
\label{fig:cross1}
\end{figure}

 If the relation (\ref{2-2}) holds for the dynamical system generated by (\ref{mfunc}) then the graphics crosses with  the line $y=\frac{1}{2}$ two times, if (\ref{2-1}) holds then it crossed with that line one time (see Figures \ref{fig:cross} and \ref{fig:cross1}). We use these properties later.

\subsection{The case $a\in(0,1]$, $a\leq b\leq \frac{4a}{4-a^2}$}

\begin{prop}\label{prop1}
If the dynamical system generated by the function (\ref{mfunc}) satisfies the condition (\ref{2-2}), then the dynamical system has no odd periodic points.
\end{prop}

We use the following Lemma to prove Proposition \ref{prop1}.

\begin{lem}\label{lem2}
If the dynamical system associated by the function (\ref{mfunc}) satisfies (\ref{2-2}) then  the followings hold for sets $A_1$, $A_2$, $A_3$, $A_4$:
\begin{itemize}
\item $f(A_1)\subset A_2\cup A_3$;
\item $f(A_2)\subset A_4$;
\item $f(A_3)\subset A_1$;
\item $f(A_4)\subset A_2\cup A_3$.
\end{itemize}
Where
$A_1=(\frac{1}{2}-\frac{b}{4};(\frac{1}{2}+\frac{a}{4})(1-\frac{b}{2}+\frac{ab}{4})(1-\frac{b}{2}+b(\frac{a-b}{4}+\frac{a^2b}{16}))]$,\\[3mm]
$A_2=((\frac{1}{2}-\frac{b}{4})(1+\frac{a}{2}+\frac{ab}{4});\frac{1}{2}]$, \ \ \ $A_3=(\frac{1}{2},(\frac{1}{2}+\frac{a}{4})(1-\frac{b}{2}+\frac{ab}{4})]$,\\[3mm]
$A_4=((\frac{1}{2}-\frac{b}{4})(1+\frac{a}{2}+\frac{ab}{4})(1+\frac{a}{2}-a(\frac{a-b}{4}-\frac{ab^2}{16}));\frac{1}{2}+\frac{a}{4};]$.\\
\end{lem}

\begin{center}
\tikz[xscale=15]{
\draw [help lines,xstep=0.5cm] (0,0) grid (1,0);
\draw [line width=0.5mm, blue] (0,0) plot [domain=0.375:0.383] (\x,0);
\draw [line width=0.5mm, red] (0,0) plot [domain=0.492:0.5] (\x,0);
\draw [line width=0.5mm, black] (0,0) plot [domain=0.5:0.508] (\x,0);
\draw [line width=0.5mm, orange] (0,0) plot [domain=0.617:0.625] (\x,0);
\draw [->] (0,0) -- (1,0) ;
\node at (0.375,0.3) {$A_1$};
\node at (0.47,0.3) {$A_2$};
\node at (0.53,0.3) {$A_3$};
\node at (0.617,0.3) {$A_4$};
\foreach \y in {0,0.1,0.2,0.3,0.4,0.5,0.6,0.7,0.8,0.9,1}{\node [below] at (\y,0) {\tiny \y};}
}
\end{center}

The proof of Lemma is derived from finding the images of $f_{a,b}$, above.

\begin{proof} Suppose the point $x$  belongs to the set $A_1$. According to Lemma \ref{lem2} $f(x)\in A_2$ or $f(x)\in A_3$.

a) If $f(x)\in A_2$ then $f^2(x)\in A_4$. $f^3(x)\in A_2$ or $f^3(x)\in A_3$.

a1) If $f^3(x)\in A_2$ then $f^4(x)\in A_2$ and it is repeated.

a2) If $f^3(x)\in A_3$ then $f^4(x)\in A_1$.

Consequently if $x\in A_1$  is a periodic point then this point
must be either 2-periodic or 4-periodic point.
That's why, there is no odd periodic point in $A_1$.
We can show the same for the remaining 3 sets.

 Now let's assume that $x$  doesn't belong to these four sets.
 Then the point must be either in one of sets $K_1=(\frac{1}{2}-\frac{b}{4},\frac{1}{2}]\setminus (A_1\cup A_2 )$ or $K_2=(\frac{1}{2},\frac{1}{2}+\frac{a}{4}]\setminus (A_3\cup A_4 )$. The following two cases may hold. First,
 the image of $f$ at $x$ belongs to four sets in above after several times. Then the proof is clear.
 Second, if it remains in these two sets even after finite or countable steps, then there is no odd periodic points. Because $x\in K_1$, $f(x)\in K_2$, $f^2(x)\in K_1$ and so on. If $x\in K_2$ then $f(x)\in K_1$, $f^2(x)\in K_2$, $f^3(x)\in K_1$, $\ldots$ .

\end{proof}
The following example shows that if the condition of Proposition \ref{prop1} is not satisfied then
there may exist an odd periodic point.

{\bf Example.} We have 3-period points when the function is
\begin{equation}\label{pfunc1}
f_{\frac{1}{2},1}(x)=\left\{
     \begin{array}{ll}
       x(\frac{3}{2}-\frac{x}{2}), & \hbox{$0\leq x\leq \frac{1}{2}$;} \\
       x^2, & \hbox{$\frac{1}{2}<x\leq1$.}
     \end{array}
   \right.
\end{equation}
Computer analysis shows that the points $0,47556611\ldots$, $0,60026760\ldots$, $0,36032119\ldots$ are 3-periodic points of (\ref{pfunc1}).

\section{Lyapunov exponents and bifurcation}

Recall\footnote{https://en.wikipedia.org/wiki/Lyapunov$_-$exponent} that for discrete time dynamical
 system $ x_{{n+1}}=f(x_{n})$, for an orbit starting with $x_{0}$
  the  maximal \textbf{Lyapunov exponent}  can be defined as follows:
\begin{equation}\label{exp}
\lambda (x_{0})=\lim _{n\to \infty }{\frac {1}{n}}\sum _{i=0}^{n-1}\ln |f'(x_{i})|.
\end{equation}

The Lyapunov exponent "$\lambda$", is useful for distinguishing among the various types of
orbits. It works for discrete as well as continuous systems.

\begin{prop}\label{prop3}
Let $x_0 \in [0,1]\setminus\{\frac{1}{2}\}$. Then the Lyapunov exponent is non negative for (\ref{mfunc}).
\end{prop}
\begin{proof}
According to Theorem \ref{thm1} we have $f'(x)>1$ for all $x\in [0,1]\setminus\{\frac{1}{2}\}$. So, one gets $\ln |f'(x_{i})|>0$ which yields that
the limit (\ref{exp}) is non negative.
\end{proof}

\begin{figure}[h]
    \includegraphics[width=7cm]{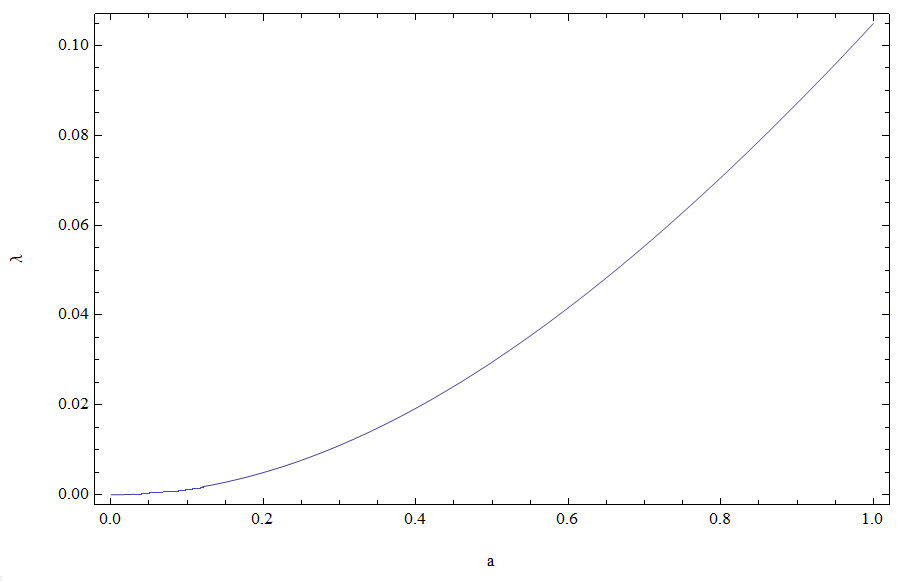}
    \includegraphics[width=7cm]{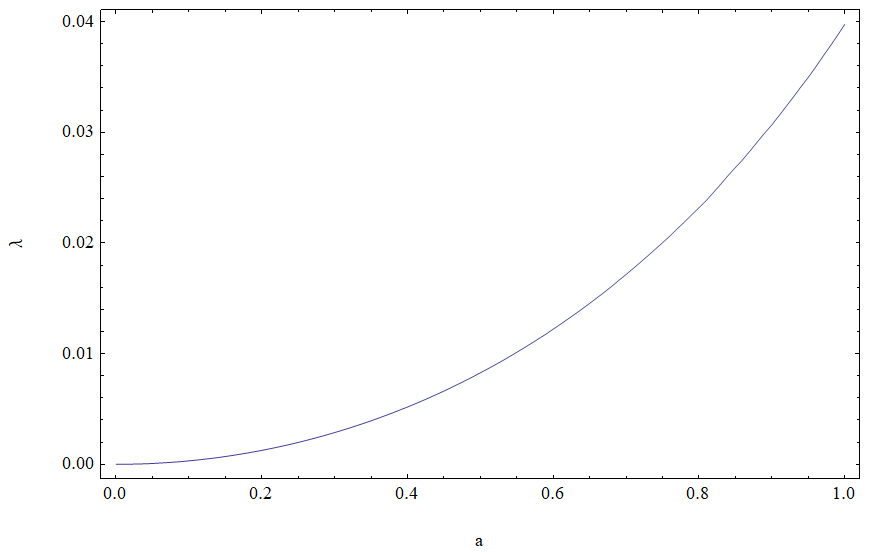}
    \caption{Typical graphics of the Lyapunov exponent for the system function $f_{a,b}(x)$
    for different values of the parameter $b$ related to $a$: left $b=a$ and right $b=\frac{a}{4-a^2}$.}
  \label{fig:Lyapunovexp}
\end{figure}

It is known that if $\lambda>0$ then the orbit is unstable and chaotic.
Nearby points, no matter how close, will diverge to any arbitrary
separation. All neighborhoods in the phase space will eventually be visited. These points are said to be
unstable (see Figure \ref{fig:Lyapunovexp}).

Most commonly applied to the mathematical study of dynamical systems, a \textbf{bifurcation}\footnote{https://en.wikipedia.org/wiki/Bifurcation$_-$theory} occurs when a small smooth change made to the parameter values (the bifurcation parameters) of a system causes a sudden 'qualitative' or topological change in its behavior. Bifurcations occur in both continuous systems (described by ODEs, DDEs or PDEs) and discrete systems (described by maps).

From Figure \ref{bifurcation1} to Figure \ref{bifurcation5} we show the bifurcation diagrams for the system associated with (\ref{mfunc}). A \textbf{bifurcation diagram}\footnote{https://en.wikipedia.org/wiki/Bifurcation$_-$diagram} shows the values visited or approached asymptotically (fixed points, periodic orbits, or chaotic attractors) of a system as a function of a bifurcation parameter in the system. It is usual to represent stable values with a solid line and unstable values with a dotted line, although often the unstable points are omitted.

In Figure \ref{bifurcation1} the limit set consists of three intervals as we saw in Proposition \ref{prop1} for each value of the parameter $a$ in $[0,1]$ (The sets $A_2$ and $A_3$ are symmetric about $\frac{1}{2}$). Because if $b=a$ then the condition (\ref{2-2}) holds. That's, the trajectory of any initial point in $[0,1]$ (except periodic points) is dense in these intervals.
In Figure \ref{bifurcation1b}  the limit set consists of four intervals for each value of $a$.

The boundary line changes according to how $b$ depends on $a$. For example, from Figure \ref{bifurcation1} to \ref{bifurcation4} $b$ is linearly depended to $a$, so the boundary line is linear. But, in Figure \ref{bifurcation5} it is a curve, because $b$ isn't linearly depended to $a$.\\

\section{Biological interpretations}

 Let $x = (x_1, x_2)\in S^1$ be an
initial state, i.e. the probability distribution on the set $E=\{1, 2\}.$

The following are interpretations of our results:

\begin{itemize}
\item If $a=0$ then by (\ref{vabp}) we have $P_{12,1}=1/2$ and Proposition \ref{pj1} means that
if initially the population has the specie 1 with a probability $x\leq 1/2$ then the state of the population
does not change, i.e. it remain the same as $(x,1-x)$. In case when initially the specie 1 has probability $x>1/2$
then after finitely many generations the population comes to an equilibrium state $(p, 1-p)$ with $p\leq 1/2$. Part 3)
of Proposition \ref{pj1} means that for each equilibrium state the past of such state of population can be uniquely reproduced.
\item \emph{Both species survive}: Proposition \ref{prop2} and Lemma \ref{lj1} means the specie 1 (resp. 2) of the population asymptotically survives with probability
$x\in A=(\frac{1}{2}-\frac{b}{4},\frac{1}{2}+\frac{a}{4}]$, i.e.,
${1\over 4}<{1\over 2}-{b\over 4}<x< \frac{1}{2}+\frac{a}{4}<{3\over 4}$ (resp.  for $y=1-x$ we have
${1\over 4}<{1\over 2}-{a\over 4}<y< \frac{1}{2}+\frac{b}{4}<{3\over 4}$). Thus both survive with probability $>1/4$.
\item Theorem \ref{tj1} means that the population may have an alternating (periodic) state. But Theorem \ref{thm1} says
that there is no attracting periodic state of the population. Consequently, under condition $ab>0$, if initially, the
population is not on at equilibrium or periodic state then it
does not tend to such a state with the passage
of time (non-stable fixed and periodic points).

\item Proposition \ref{prop1} means in general that the population does not come back to the initial state after
spending an odd time (odd generation).

\item Proposition \ref{prop3}: the nearby states of the population, no matter how close, will diverge to a
separation.

\end{itemize}

\begin{figure}[h]
    \includegraphics[width=8cm]{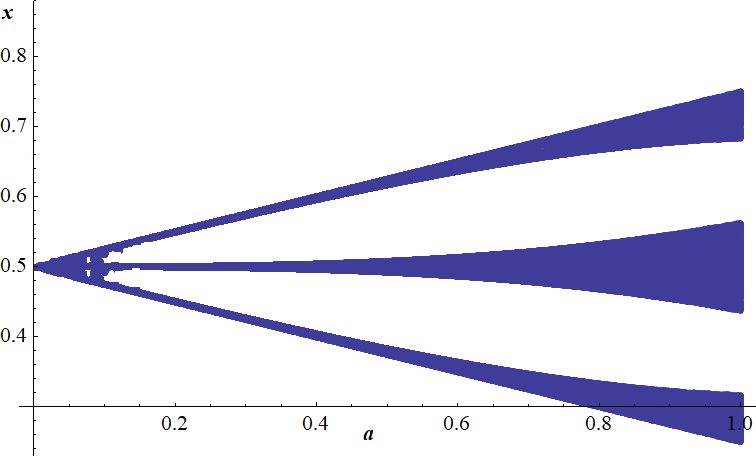}
     \caption{Bifurcation diagram for the system function $f_{a,b}(x)$ for the case $b=a$.}
      \label{bifurcation1}
\end{figure}
\begin{figure}[h]
    \includegraphics[width=8cm]{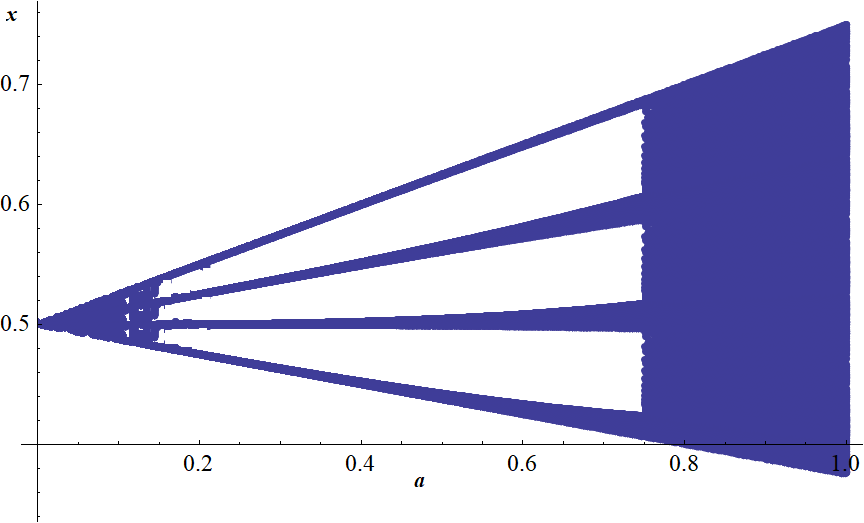}
    \caption{Bifurcation diagram for the system function $f_{a,b}(x)$ for the case $b=\frac{a}{2}$.}
    \label{bifurcation1b}
 \end{figure}

\begin{figure}[h]
    \includegraphics[width=8cm]{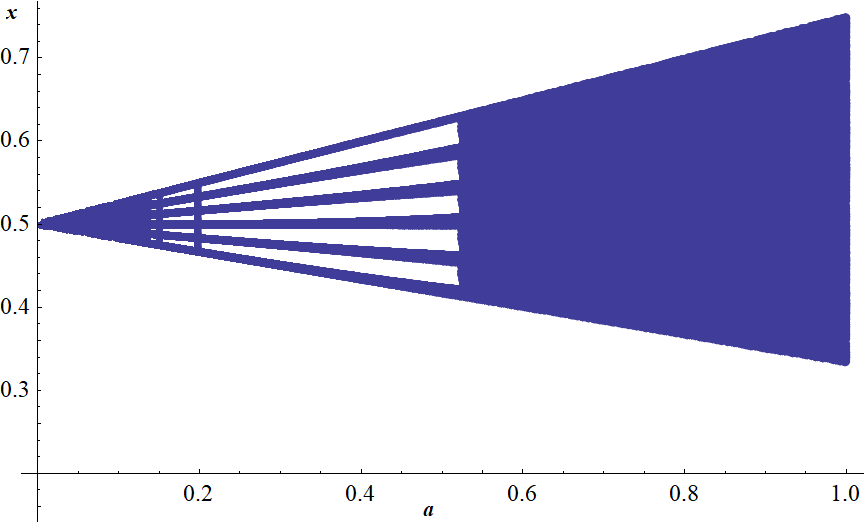}
    \caption{Case  $b=\frac{2a}{3}$}
\end{figure}
\begin{figure}[h]
    \includegraphics[width=8cm]{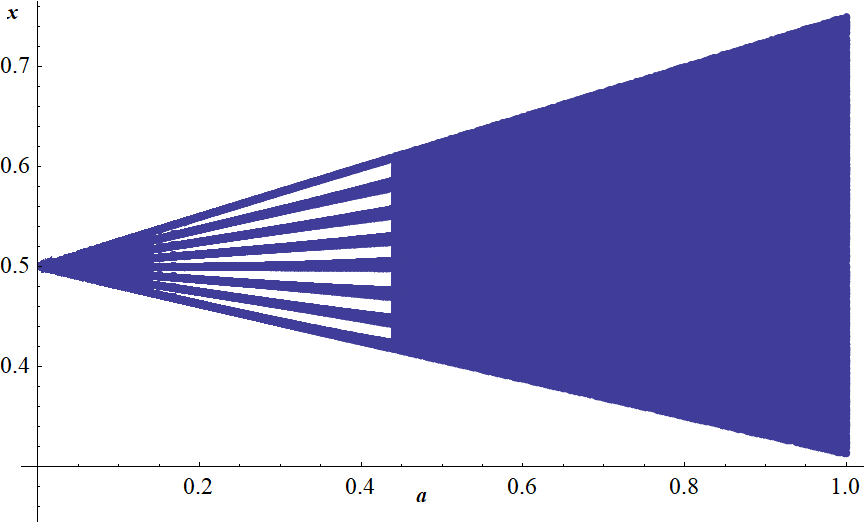}
       \caption{Case $b=\frac{3a}{4}$.}
  \label{bifurcation3}
\end{figure}

\begin{figure}[h]
    \includegraphics[width=8cm]{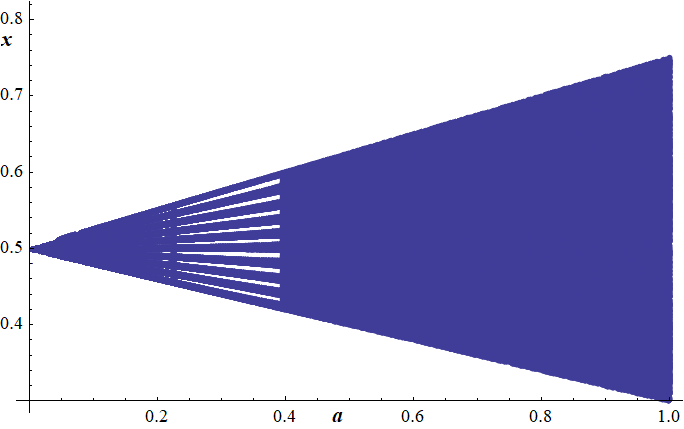}
    \caption{Case $b=\frac{4a}{5}$.}
\end{figure}
\begin{figure}[h]
 \includegraphics[width=8cm]{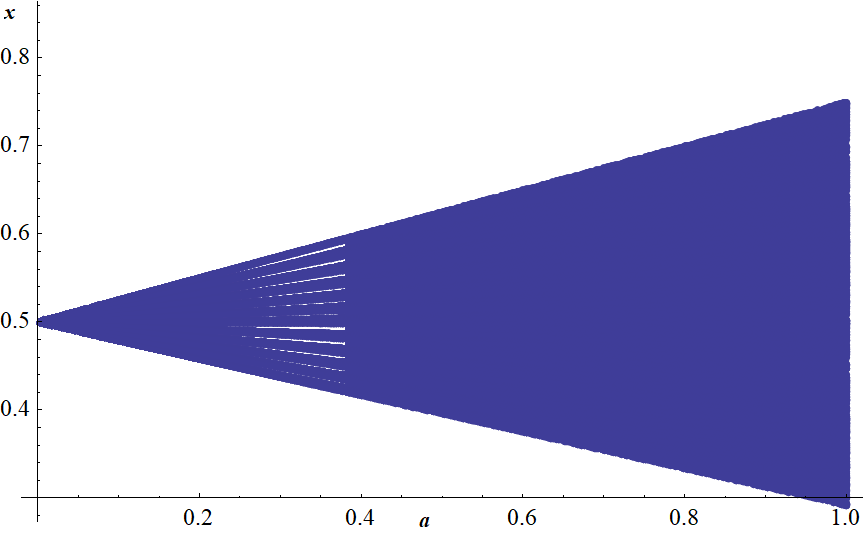}
 \caption{Case $b=\frac{5a}{6}$.}
  \label{bifurcation4}
\end{figure}

\begin{figure}[h]
    \includegraphics[width=8cm]{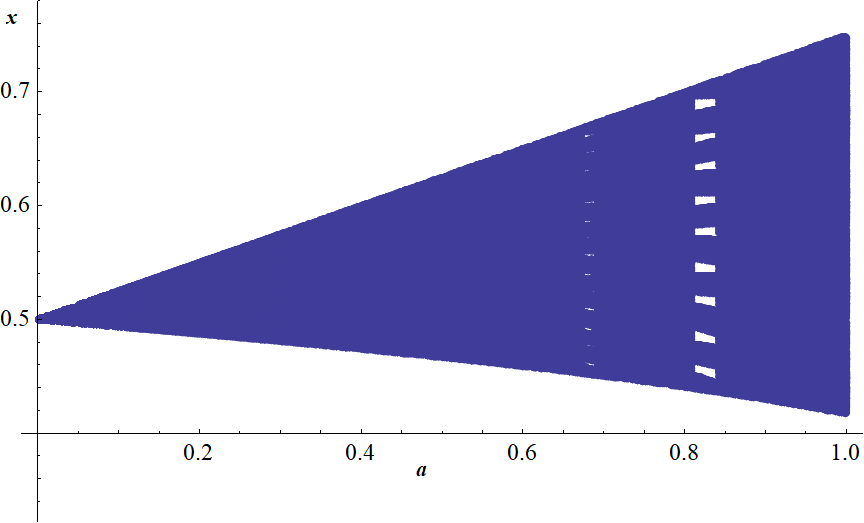}
    \caption{Case $b=\frac{a}{4-a^2}$}
 \end{figure}
\begin{figure}[h]
    \includegraphics[width=8cm]{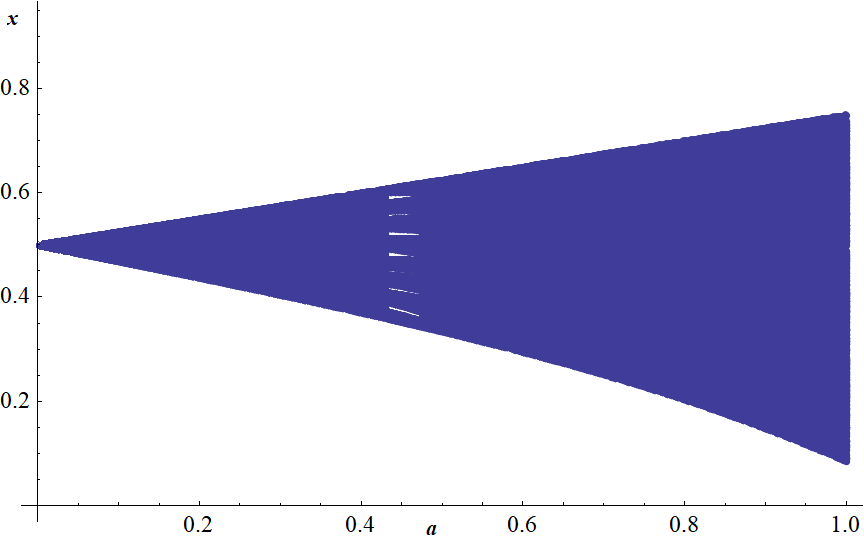}
  \caption{Case $b=\frac{5a}{4-a^2}$.}
  \label{bifurcation5}
\end{figure}


\begin{thebibliography}{10}
\bibitem{ARS} S. Albeverio, U.A. Rozikov, I.A. Sattarov, \textit{$p$-adic $(2,1)$-rational dynamical systems.} Jour. Math. Anal. Appl., \textbf{398}(2),(2013) 553--566.

\bibitem{ref2} M. di Bernardo, C.J. Budd, A.R. Champneys, P. Kowalczyk, \textit{Piecewise-smooth Dynamical Systems: Theory and Applications}. Applied Math. Sci., \textbf{163} (2008).

\bibitem{ref3}  R.L. Devaney, \textit{An introduction to chaotic dynamical system},
Westview Press, 2003.

\bibitem{G2} R. N. Ganikhodzhaev, R. E. Abdirakhmanova, \textit{Fixed and periodic points of
quadratic automorphisms of non-Volterra type}, Uzbek Math. J. \textbf{2} (2002) 6--13, (Russian).

\bibitem{G4} R. N. Ganikhodzhaev, D. B. Eshmamatova, \textit{Quadratic automorphisms of simplex
and asymptotical behavior of their trajectories}, Vladikavkaz Math. \textbf{8} (2006) 12--28.

\bibitem{GMR} R.N. Ganikhodzhaev, F.M. Mukhamedov, U.A. Rozikov, \textit{Quadratic stochastic operators and processes: results and open problems}.  Inf. Dim. Anal. Quant. Prob. Rel. Fields., \textbf{14}(2) (2011), 279--335.

\bibitem{GGJ} N.N. Ganikhodjaev, R.N Ganikhodjaev (Ganikhodzhaev), U.U. Jamilov (Zhamilov), \textit{Quadratic stochastic operators and zero-sum game dynamics}, Ergodic Theory Dyn. Syst. \textbf{35}(5), 1443--1473.


\bibitem{Ke} H. Kesten, \textit{Quadratic transformations: A model for population growth}, I, II, Adv.
Appl. Probab. \textbf{2}(2) (1970) 1-82; 179-228.

\bibitem{M} F. Mukhamedov, M.Saburov, A.H.M. Jamal, \textit{On dynamics of $\xi^s$ quadratic stochastic operators}, Inter. Jour. Modern Phys.: Conf. Ser.,
\textbf{9}, (2012), 299--307.

\bibitem{ME} F. Mukhamedov, A. F. Embong, \textit{On $b$-bistochastic quadratic stochastic operators}, Jour. Inequalities and Appl. 2015, Article number: 226 (2015)

\bibitem{MG} F.M. Mukhamedov, N.N. Ganikhodjaev, \textit{Quantum quadratic operators and processes}. Lecture Notes in Mathematics, 2133. Springer, Cham, 2015.

\bibitem{L} Y.I. Lyubich, \textit{Mathematical structures in population
genetics}, Springer-Verlag, Berlin, 1992.


\bibitem{LRS} A.R. Luna,  U.A. Rozikov, I.A. Sattarov, \textit{$p$-adic dynamical systems of $(3,1)$-rational functions with unique fixed point}. 
arxiv:1807.11561, (2018).

  \bibitem{RS1} U.A. Rozikov, N.B. Shamsiddinov, \textit{On Non-Volterra Quadratic
Stochastic Operators Generated by a Product Measure.} Stoch.
Anal. Appl., \textbf{27}(2)  (2009),  353--362.

\bibitem{RZ} U.A. Rozikov, A. Zada, \textit{On $\ell$- Volterra quadratic stochastic
operators.} Inter. Journal Biomath. \textbf{3}(2) (2010),
143--159.

 \bibitem{RSh} U.A. Rozikov, S.K. Shoyimardonov, \textit{On ocean ecosystem discrete time dynamics
generated by $\ell$-Volterra operators.}  Inter. Jour. Biomath. \textbf{12}(2) (2019), 1950015-24.

\bibitem{RJ} U.A. Rozikov, U.U. Zhamilov, \textit{On dynamics of strictly
non-Volterra quadratic operators on two-dimensional simplex}. Sbornik: Math.
 \textbf{200}(9),  (2009), 1339--1351.


\bibitem{RS} U.A. Rozikov, I.A. Sattarov , \textit{$p$-adic dynamical systems of $(2,2)$-rational functions with unique fixed point}. Chaos, Solitons \& Fractals, \textbf{105} (2017), 260-270.

\bibitem{ref1} U.A. Rozikov, I.A. Sattarov, J.B. Usmonov, \textit{The Dynamical System Generated by the Floor Function $[\lambda x]$}. Jour. Appl. Nonlinear Dyn., \textbf{5}(2) (2016), 185-191.

\bibitem{RSY} U.A. Rozikov, I.A. Sattaror, S. Yam, \textit{$p$-adic Dynamical Systems of the Function $ax/x^2+a$}. $p$-adic Numbers, Ultrametric Analysis and Applications, \textbf{11}(1), (2019) 77-87.

\bibitem{Rb} U.A. Rozikov, \textit{An introduction to mathematical billiards}. World Scientific Publishing Co. Pte. Ltd., Hackensack, NJ, 2019.

\bibitem{ref4} J.B. Usmonov, \textit{On a two dimensional dynamical system generated by the floor function}. Uzbek Mathematical Journal, 2 (2019), 127-134.

\end{thebibliography}
\end{document}